\documentclass[a4paper,11pt]{amsart}
\usepackage[cp1250]{inputenc}
\usepackage{epic}
\usepackage{amssymb,amsmath,amsthm}
\usepackage{xcolor,enumerate,verbatim}
\newtheorem{thm}{Theorem}[section] \newtheorem{lem}[thm]{Lemma}
\newtheorem{wnioch}[thm]{Corollary} \newtheorem{obs}[thm]{Observation}
\newtheorem{prop}[thm]{Proposition}
\newtheorem{prob}[thm]{Problem}
\theoremstyle{definition}
\newtheorem{deff}[thm]{Definition} \newtheorem{przy}[thm]{Example} \newtheorem{rem}[thm]{Remark}
\newcommand{\N}{\mathbb{N}} \newcommand{\R}{\mathbb{R}}\newcommand{\Q}{\mathbb{Q}}
\newcommand{\RbezQ}{\mathbb{R\setminus Q}}\newcommand{\mcU}{\mathcal{U}}
\newcommand{\TcompH}{ T_{2\frac{1}{2}} }
\newcommand{\cl}[1]{\overline{#1}}
\title{Locally ordered topological spaces}
\keywords{order topology; linear order; separation axioms; connected space; local connectedness;
compact space; Lindel\" of space; hereditary property;}
\subjclass[2010]{Primary 54F05; Secondary 06F30, 54D10, 54E99}

\author{Piotr Pikul}
\address{P. Pikul, Instytut Matematyki, {Wydzia\l} Matematyki i Informatyki,
Uniwersytet Jagiello\' nski, ul. \L ojasiewicza 6, 30-348 Krak\' ow, Poland}
\email{Piotr.Pikul@im.uj.edu.pl}
\usepackage{hyperref}

\begin{document}
\begin{abstract}While topology given by a linear order has been extensively studied,
this cannot be said about the case when the order is given only locally. The aim of this paper
is to fill this gap. We consider relation between local orderability and separation
axioms and give characterisation of all connected, locally connected or compact
locally ordered Hausdorff spaces.
A collection of interesting examples is also offered.
\end{abstract}
\maketitle
\section{Introduction}
Formal definition of an ordered set appeared in 1880 due to C{.} S{.} Pierce but the idea
was somehow present in mathematics and philosophy long before. While mentioning early research
on ordered sets names of Dedekind and Cantor cannot be omitted. Some historical notes can be found
in \cite{nachbin}. 

The concept of order topology appeared probably at the same time as abstract topology itself.
On the one hand it is a very classical notion today, on the other, there were studied various
different connections between topology and order on a set (see e.g. \cite{closedin} or \cite{nachbin}).
Several notable results concerning linearly ordered spaces and their subspaces
were listed in \cite{history}.
For more recent works the reader is referred to \cite{recprog}.

The natural notion of a locally ordered topological space, which we consider in this article,
seem to had appeared only once, in not easily accessible dissertation by Horst Herrlich
\cite{herrlich}, where the main concern was whether a space is orderable.
Several results are similar to those presented in this article; however, we omit the
notion of end-finite space (i.e. space with at most two non-cutpoints
for each connected component) and put some attention to the case when a locally
ordered space is not orderable.

The aim of this paper is to present general results concerning locally ordered spaces
and the most classical topological notions.
Our survey starts with basic definitions and observations concerning
separation axioms and hereditarity of local orderability.
Then we pass to properties of connected and locally connected spaces and prove
characterisation of all connected $T_3$ locally ordered spaces (Theorem~\ref{thm.class}).
This leads also to description of both locally connected (Corollary~\ref{loc-conn})
and Lindel\"of (Theorem~\ref{thm.lind}) among such spaces.
Provided characterisation is valid also for arbitrary connected or compact subsets
of a locally ordered $T_3$ space.

All notable examples of locally ordered spaces are presented in the separate section.
Some of them are well known to topologists but possibly not for their local orderability.

\subsection{Notation and terminology}
We are not going to denote topological spaces formally as pairs $(X,\tau)$.
Since we never refer to two different topologies on one set at a time,
the risk of ambiguity is minimal.

As we will see at the very beginning, all spaces of concern would be $T_1$
hence we are not going to distinguish e.g. between ``$T_4$ space'' and ``normal space''.
By \emph{Urysohn space} ($\TcompH$) we mean space in which every two distinct points have
neighbourhoods with disjoint closures.
We call a space \emph{completely Hausdorff} if its any two distinct points can be
separated by a real-valued function.
By \emph{semiregular space} we mean Hausdorff space which has a basis consisting of sets being
interiors of their closures.

When considering ordered sets we refer to strict (irreflexive) linear order relations (cf. \cite{engel}).
In most cases we do not denote the ordering relations explicitly, similarly like the topology on a set.
We write ``K is an open interval with respect to the order on U'' as long as
it is a sufficient clarification. The natural notation ``$(a,b)_U$'' for
an (open) interval in $U$ is also used. Symbols like ``$(\leftarrow,b]$'' and
``$(a,\rightarrow)$'' denote unbounded intervals.

By a closed interval we mean an interval including both its least and its greatest element
-- not any interval which is a closed set. For open intervals there is no such ambiguity.

A~subset of a linearly ordered set 
is called \emph{convex} if for its every two points it contains
the interval spanned by them. The term \emph{endpoint} stands for the superemum
or the infimum of a convex set (they may not exist).

\section{Basic definitions and properties}
\subsection{Order topology}
First let us recall basic facts about classical order topology.
They belong to the folklore and are mostly mentioned as exercises (see e.g. \cite{engel}, \cite{willard}).

Given linearly ordered set $(X,<)$, by \emph{order topology}
(called also an open interval topology) we mean a topology defined
by the basis consisting of open intervals (including unbounded) in $(X,<)$.
Space $X$ with an order topology is called \emph{linearly ordered (topological) space}
or \emph{orderable space} (if the order on $X$ is not fixed).
We call a linear order on $X$ \emph{compatible with the topology} if
the associated order topology equals the topology on $X$.

Every linearly ordered (orderable) space is hereditarily normal ($T_5$) -- even hereditarily
collectionwise normal (\cite{colwise}).

Every connected subset of a linearly ordered topological space has to be convex.
Closure and interior of a convex set are convex.

A~linear ordering $<$ on $X$ is called \emph{continuous} if it is dense and every
convex set is an interval (possibly unbounded)\footnote{The standard order on rational numbers
is dense, but the convex set $\{x\in \mathbb{Q}: x^2<2\}$ is not an interval in $\mathbb Q$.}.
This is equivalent to the connectedness of associated order topology (cf. \cite[Problem 5.3.2]{engel}).
A subset of a connected linearly ordered space is connected if and only if it is an interval.

On a connected orderable space containing at least two points, there are precisely two
linear orders (each one is the reverse of the other) compatible with the topology.
Although the fact is long-known (e.g. \cite[II.]{herrlich}), we provide an ``exceptionally topological''
proof in the Appendix (Theorem~\ref{thm.con-ord}). 

Compactness of order topology is equivalent to the existence of supremum and infimum 
for any subset (cf. \cite[Problem 3.12.3 (a)]{engel}).
For a connected space it is enough to check whether it has the smallest and the greatest element.
Every connected linearly ordered space is automatically locally connected and locally compact.
On arbitrary compact subset of a linearly ordered space the subspace topology and the
order topology given by the linear order inherited from the original space coincide
(\cite[I. Satz 13c]{herrlich}).
In particular orderability is hereditary on compact subspaces.

\subsection{Locally order topology}
Now we can pass to main definitions of this article.

\begin{deff}
Let $X$ be a~topological space. We say $X$ is a \emph{locally ordered topological space}
(or \emph{has locally order topology}) if each point in $X$ has an orderable neighbourhood.
An open cover of $X$ consisting of linearly ordered sets with fixed linear orders will be called
an \emph{atlas of orders}.

We say $X$ is \emph{regularly locally ordered} if each point in $X$ has a neighbourhood
whose closure is orderable. Then a \emph{regular atlas of orders} is an open cover of $X$
together with fixed linear orders on the closures.
\end{deff}
Note that at the beginning we do not assume that considered spaces satisfy any separation axiom.

Let us make some basic observations. 
\begin{obs}\mbox{ } \begin{enumerate}[\ 1.]
\item Every space with order topology is regularly locally ordered. 
\item Every one dimensional topological manifold is regularly locally ordered.
\end{enumerate}
\end{obs}

In particular a~circle is locally ordered space but not orderable at the same time
and cannot be embedded in a linearly ordered space.
Later we are introducing a~whole class of spaces sharing those properties.

The following theorem is also a simple observation.
\begin{thm}
Every open subspace of locally ordered space is itself locally ordered.
\end{thm}
\begin{proof}
Let us observe that every interval in a~space with order topology is itself a linearly ordered
topological space (order topology and subspace topology coincide). Since open subspace for
each of its points contains some interval with respect to the order on a~neighbourhood,
it satisfies the definition of local orderability.
\end{proof}

This behaviour is different from the case of linearly ordered spaces,
e.g. $[0,1]\cup\{2\}$ with Euclidean topology is a~linearly ordered space
containing a~not orderable open subset  $(0,1)\cup\{2\}$. 

The property of local orderability is not hereditary in general.
\begin{przy}\label{przy.nieLokPorz} 
The set $(-1,0]\cup\bigcup_{n=1}^\infty
\left( \{\frac{1}{3n}\} \cup\bigl( \frac{1}{3n-1},\frac{1}{3n-2} \bigr) \right)$
with the topology induced from $\R$ is not locally ordered.
\end{przy}
\begin{proof}
Assume $0$ has some linearly ordered neighbourhood $U$. Clearly the connected component of~$0$
(an interval of the form $(s,0]$) has to be some interval closed at~$0$. Without loss of generality
we can assume that it consists of elements not greater than~$0$ with respect to the order on~$U$.
The rest of $U$ has countably many connected components, namely singletons and sets homeomorphic
to $(0,1)$. 
Both isolated points and intervals converge to $0$.
There exist such singleton and interval that there is no element between them, because otherwise
there will be an infinite set of intervals or singletons between some two leading to contradiction
with their convergence to $0$. Hence the singleton must be in the closure of open interval
while it is not the case.
\end{proof}

\begin{rem}
The example was known to Herrlich (\cite{herrlich}). A different idea is presented as a part
of Example~\ref{przy.nLPPT.skT2}.
\end{rem}

One can notice that the above space is a closed subset of some orderable subspace of $\R$.
It will be shown later (Proposition~\ref{prop.compsub}) that local orderability is
hereditary on compact subspaces.

At this point we know that there is no inclusion between the classes of locally ordered spaces and
generalised ordered spaces (suborderable spaces).

The following lemmas deal with separation axioms and also explain why stronger version
of local orderability property is called \emph{regular} local orderability.
\begin{lem} Every locally ordered space is $T_1$.\end{lem}
\begin{proof}
Given two distinct points $x,y$ of a locally ordered space, either $y$ does not
belong to an ordered neighbourhood of $x$ from the atlas of orders or we can use the fact
that an ordered neighbourhood of $x$ is $T_1$ itself.
\end{proof}

\begin{lem}\label{lematT3}
For a locally ordered space $X$ the following conditions are equivalent:
\begin{enumerate}[(a)]
\item $X$ is regularly locally ordered
\item $X$ is regular ($T_3$)
\item $X$ is Tychonoff ($T_{3\frac{1}{2}}$).
\end{enumerate} \end{lem}
\begin{proof}
The implication (c)$\Rightarrow$(b) is simple and well known.\smallskip

(a)$\Rightarrow$(c) Fix $A$, a~closed subset of a regularly locally ordered space $X$,
and a~point $x\in X\setminus A$. Let $U$ be a~neighbourhood of $x$ such that $\cl{U}$
is orderable. Linearly ordered spaces are $T_3$ (even $T_5$), so we can find
disjoint open sets $V$ and $W$ in $\cl{U}$ such that $x\in V$ and $A\cap U \subseteq W$.
Note that $\cl{U} \setminus W$ is closed in $X$.

The set $V\cap U$ is then a~neighbourhood of $x$ in $X$
disjoint with $X\setminus (\cl{U} \setminus W)$, a~neighbourhood of $A$.
We can pick any function on $\cl{U}$ separating $x$ and $\cl{U}\setminus V$ (since 
lineraly ordered spaces are Tychonoff) and extend it with
a constant on $X\setminus\cl{U}$. 
\smallskip

(b)$\Rightarrow$(a) Fix $x$, a point in a locally ordered $T_3$ space. It has an orderable
neighbourhood $U$. Pick an open set $V$, such that $x\in V\subseteq \cl{V}\subseteq U$.
If we find then an open interval inside $V$ containing $x$, its closure will be
contained in $U$, and hence will be an interval which is orderable.
\end{proof}

\begin{wnioch}Every locally ordered $T_3$ space is completely Hausdorff.\end{wnioch}
The above implication cannot be reversed.

For locally ordered spaces there are no other implications between low
separation axioms (namely: $T_1$, $T_2$, $\TcompH$, semiregularity and complete Hausdorff)
than those valid for topological spaces in general.
The table below lists the sufficient counterexamples:
\begin{center}\begin{tabular}{|r|c|c|c|c|c|}
\hline
Example \rule{0pt}{2.5ex}& \makebox[2em]{\ref{przy.nieT2}}
&\makebox[2em]{\ref{przy.nieUr}}&\makebox[2em]{\ref{przy.nCompH}}&
\makebox[2em]{\ref{przy.nielc}}&\makebox[2em]{\ref{przy.nieT3}}\\
\hline Hausdorff \rule{0pt}{2.3ex}& 		0 & 1 & 1 & 1 & 1 \\
\hline Urysohn \rule{0pt}{2.2ex}&			0 & 0 & 1 & 1 & 1 \\
\hline completely Hausdorff \rule{0pt}{2.2ex}&	0 & 0 & 0 & 1 & 1 \\
\hline semiregular \rule{0pt}{2.2ex}&		0 & 1 & 1 & 0 & 1 \\
\hline $T_3$ \rule{0pt}{2.2ex}&				0 & 0 & 0 & 0 & 0 \\ \hline
\end{tabular}\end{center} 
All mentioned examples are connected locally ordered topological spaces.

Basing on the Example~\ref{przy.nielc} one can modify spaces from Examples~\ref{przy.nieUr}
and \ref{przy.nCompH} to be not semiregular.

Since both local orderability and $T_3$ are hereditary on open subspaces,
so is regular local orderability. Further, we will show that
regular local orderability is also hereditary on connected (Lemma~\ref{lem.connected})
and compact (Proposition~\ref{prop.compsub}) subsets.

\begin{rem}[On higher separation axioms]\mbox{ }
\begin{enumerate}[\ 1.]
\item A regularly locally ordered space need not to be normal
(see Examples \ref{przy.rat} and \ref{deska}).

\item The long line (see e.g. \cite{counter}) is an example proving that a $T_5$
locally ordered space need not to be $T_6$ (not every closed set is $G_\delta$).

\item A $T_6$ linearly ordered space is first-countable and the opposite implication
is not true (e.g. the long line). This fact easily implies that a $T_6$ locally ordered space
is first-countable.

\item Second-countable locally ordered space may not be even Hausdorff (Example~\ref{przy.nieT2}).
First-countable reagularly locally ordered space may not be normal (Example~\ref{przy.rat}).
For regularly locally ordered spaces the second axiom of countability is strong assumption
implying the decomposition described in Corollary~\ref{wn-2AP}.

\item It is also known that linearly ordered spaces which are $T_6$ may not be metrizable.
As an example one can take the set $[0,1]\times[0,1]$ with the lexicographic order
(see \cite{counter}).
\end{enumerate}
\end{rem}

Lutzer (\cite{lutz}) proved that a linearly ordered space $X$ is metrizable if and only if
the diagonal ($\{(x,y)\in X\!\times\! X:x=y\}$) is a $G_\delta$ subset of the product space $X\times X$.
The Example~\ref{przy.rat} shows that such condition is not sufficient
for regularly locally ordered spaces.
It obviously implies local metrizability, hence, keeping in mind well known metrization theorem
by Smirnov (\cite{smirnov}, \cite[5.4.A]{engel}), we can formulate a following
characterisation of metrizability for locally ordered spaces.
\begin{thm}A locally ordered space is metrizable if and only if
it is paracompact and has a $G_\delta$ diagonal.\end{thm}

\section{Connectedness} 
For linearly ordered topological spaces connectedness implies local connectedness and local compactness.
This is not true in general for locally ordered spaces (see Examples~\ref{przy.nieT3} and \ref{przy.nielc}).

\begin{lem}\label{lem.loccon}\mbox{ }\nopagebreak 
\begin{enumerate}[a)]
\item A connected regularly locally ordered space is locally connected.
\item Locally connected Hausdorff locally ordered space is locally compact (and hence regular).
\end{enumerate} \end{lem}
\begin{proof}
a) Assume that space $X$ is locally ordered, $T_3$ and not locally connected. We will prove
that then $X$ is not connected.

Consider $x\in X$, a~point without a~connected, ordered neighbourhood (if there is no such point,
the space is locally connected, since connectedness of an ordered space implies its local connectedness).
Let $U$ be an ordered neighbourhood of $x$. By regularity, there exists open set $W$, such that
$x\in W\subseteq \overline W \subseteq U$. Without loss of generality we can assume
that $W$ is an interval in $U$.
Since $x$ has no connected neighbourhood, $W$ contains a~nonempty proper
subset $V$ which is both closed and open in $U$. 
Since closure of $V$ is contained in $U$, it equals the closure of $V$ in $U$.
Hence closed sets $V$ and $X\setminus V$ separate space $X$.\smallskip

b) Denote given locally connected locally ordered Hausdorff space by $X$.
Fix point $x$ and a~closed set $A$ not including $x$. There exists an ordered and connected
neighbourhood $C$ of $x$. Any closed interval in $C$ is compact and therefore closed
in the whole space. It suffices to pick from its interior any closed interval containing $x$.
\end{proof}

Hausdorff axiom is important, since $\R$ with doubled origin (Example~\ref{przy.nieT2}) is
connected and locally connected but definitely not $T_3$.

Now let us define an important class of locally ordered spaces, a generalisation of the circle.
\begin{deff} A~topological space obtained from a~compact and connected linearly ordered space
(containing at least two points) by identification of the smallest and the greatest element is
called a \emph{loop-ordered topological space}.
\end{deff}

Assuming connectedness in the definition of loop-ordered space is important, because otherwise
after the identification of the endpoints we would still obtain an orderable space.

Proofs of the following simple properties of loop-ordered spaces are left to the reader.
\begin{prop} Let $X$ be a loop ordered space. Then
\begin{enumerate}[a)]
\item $X$ is compact and connected regularly locally ordered space,
\item $X$ is hereditarily normal ($T_5$),
\item $X$ cannot be homeomorphicaly embedded in a linearly ordered space,
\item for any $x\in X$ the subspace $X\setminus\{x\}$ is connected, orderable and not compact.
\end{enumerate}
\end{prop}

\begin{rem} Due to the topological characterisation of the unit interval,
every metrizable loop-ordered space is homeomorphic to the unit circle in $\R^2$.
It can be also deduced from \cite[6.3.2c.]{engel}.\end{rem}

The follwoing simple fact would be useful in the future description of locally ordered spaces.
\begin{lem}\label{lem.clopl} In a locally ordered space every loop-ordered subspace
is open.\end{lem}
\begin{proof}
Fix a point $x_0$ in a loop ordered subspace $L$. If $U$ is an ordered neighbourhood of $x_0$,
then some connected neighbourhood of $x_0$ in $L$ has to be contained in $U$.
A connected subset of $U$ has to be convex, and $x_0$ is not its endpoint
for it is a cutpoint. Hence there exists an open interval containing $x_0$
and enclosed in $L$. Since we can find such interval for any $x_0$ in $L$,
the loop-ordered subset is open.
\end{proof}

Now we can formulate the main result of the paper, namely classification theorem
for connected locally ordered spaces.
\begin{thm}\label{thm.class}
If $X$ is a~connected regularly locally ordered topological space,
then $X$ is either an~orderable space or a~loop-ordered space.
\end{thm}

Before we prove it let us start with the following simpler case.
\begin{lem}\label{lem.union}
If a~connected Hausdorff space can be covered by two open connected
and orderable sets then it is either an orderable space or a loop-ordered space.
\end{lem}
\begin{proof}
Consider $U$ and $V$, two open connected linearly ordered subsets of $X=U\cup V$.
We can skip the trivial case when $X$ consists of less than two points.

The intersection $U\cap V$ has to be a~disjoint union of open (possibly unbounded)
intervals in $U$, namely its connected components.

For the use of this proof we say a subset $A$ of a linearly ordered space is
bounded from below (above) if there exists a strict lower (upper) bound of this set
(an element strictly smaller/greater than any element of $A$).
For example the interval $[0,1]$ is, in this sense, not bounded in itself,
while it is bounded in $\R$.

Assume there exists $W$, a~connected component of $U\cap V$ bounded from both sides
(with respect to the order on $U$). Then its closure in $U$ (closed interval) is compact,
hence closed in $X$ and therefore $\cl{W}\cap V \subseteq U$.
For $W$ is closed in the intersection $U\cap V$ we obtain $W$ is closed in $V$.
Since it is simultaneously open, $V=W\subseteq U$ and we are done.

Assuming neither $U\subseteq V$ nor $V\subseteq U$ leads then to conclusion
that $U\cap V$ consists of unbounded (from one side) intervals (with respect
to the order on $U$ as well as with respect to the order on $V$,
since the reasoning is fully symmetric).
If there is only one such interval, then we use the fact that there are only two possible
orders compatible with a~connected order topology so, by reversing order on $V$ if necessary,
we will obtain equality of orders on the intersection.
We put $U\setminus V\ni x < y \in V\setminus U$ whether it holds
$U\setminus V\ni x' < y' \in V\cap U$, or the reverse inequality otherwise.
Then we get one linear ordering on $U\cup V$, extending the one on $U$, and compatible
with the topology for it agrees locally with orders on $U$ and $V$.

Now assume that $U\cap V$ consists of two unbounded intervals. Pick any $x\in U\setminus V$.
Then $U=(\leftarrow,x ]_U\cup [x,\rightarrow )_U$ and both intervals have connected
intersection with $V$. Hence, by ``splitting point $x$''\* into $x_R$ and $x_L$,
we can apply the previous case twice to obtain that $[x_R,\rightarrow)\cup V\cup(\leftarrow,x_L]$
is a compact orderable space (connected, with both endpoints).
Identifying $x_R$ and $x_L$ leads to a loop-ordered space homeomorphic to $X$.

\* Note that there are subtle details in the operation of ``point splitting''. We define
$[x_R,\rightarrow)$ and $(\leftarrow,x_L]$ as linearly ordered spaces and observe that the
points of their intersections with $V$ have the same basic neighbourhoods as points in $X$.
Then $(x_R,\rightarrow)\cup V\cup(\leftarrow,x_L)$ equals $X\setminus\{x\}$ (together with topology).
Moreover, the identification of $x_R$ and $x_L$ leads to a point with ``the same'' basic
neighbourhoods as the point $x$. 
\end{proof}

\begin{proof}[Proof of the Theorem \ref{thm.class}]
Consider a~regular atlas of orders, $\mcU$, on $X$ consisting of connected sets.
Without loss of generality we can assume that every set from the atlas has at least two points.
Otherwise our space trivially has an order topology.

Assume there exists a~point $x_0\in X$ being an endpoint of its connected orderable
neighbourhood $U_0$. We may and do assume that it is the smallest element in the order on $U_0$.
Since $X$ is connected, for any $y\in X$ there exists a~finite
sequence of sets $U_1,\ldots,U_n$ from the atlas $\mcU$ such that $y\in U_n$ and
$U_{j-1}\cap U_j\neq \emptyset$ for $j=1,\ldots,n$.
We can inductively apply Lemma~\ref{lem.union} to sets $\bigcup_{i=0}^{j-1} U_i$
and $U_j$ to obtain that either the union $\bigcup_{i=0}^j U_i$ is orderable,
or it is a~loop-ordered space. Intervals containing $x_0$ have
one-point boundary and hence $x_0$ cannot be contained in a~loop-ordered space,
so we can proceed for every $j=1,\ldots,n$, obtaining at each step an open
orderable subspace of $X$.

We proved that every point $y\in X$ belongs to some open
and connected orderable set $U_y$ containing $x_0$. By reversing orders if necessary
we can obtain that order on each such connected set agrees with the order on $U_0$.
They form an open cover $\mathcal{V}:=\{V_y\}_{y\in X}$.
Since $x_0$ is clearly an endpoint of every set from $\mathcal{V}$,
it is also an endpoint of intersection of any $V_y, V_{y'}\in\mathcal{V}$.
We know, that $x_0$ does not belong to any loop-ordered set, hence the intersection
$V_y\cap V_{y'}$ is connected (cf. proof of the Lemma~\ref{lem.union}).
It means one of the sets is an interval in the other and the orders clearly agree.
Hence we can consider an order on $X$ being the union of all orders on sets
in $\mathcal{V}$ and such order is compatible with the topology on $X$,
since it clearly agrees locally.

Now assume that no point in $X$ is an endpoint of its connected orderable neighbourhood.
We will consider two cases.

1. There is a point $x_0\in X$ such that $X\setminus\{x_0\}$ is connected.
$x_0$ is not an endpoint of its ordered neighbourhood $U_0$, so we can pick points $a$ and $b$
from different components of $U_0\setminus \{x_0\}$. Since $X\setminus\{x_0\}$ is connected,
we can join $a$ and $b$ be a~sequence of connected orderable neighbourhoods inside $X\setminus\{x_0\}$.
Their union is not the whole space, so it cannot be a~loop-ordered space. We obtain some
open connected and orderable set $V$ containing $a$ and $b$. Then $V\cup U_0$ has to be
a~loop-ordered space, for $U_0\cap V$ has at least two connected components
(cf. proof of the Lemma~\ref{lem.union}). Connectedness implies that loop-ordered subset
has to be the whole space $X$.

2. $X\setminus\{x\}$ is not connected for any $x\in X$.
Fix $x_0\in X$. Since any neighbourhood of $x_0$ splits into at most two components,
$X\setminus\{x_0\}$ also has two components, let say $X_1$ and $X_2$.
Note that both $X_1\cup\{x_0\}$ and $X_2\cup\{x_0\}$ are connected and regularly
locally ordered and have a~point (namely $x_0$) being an endpoint of its connected
orderable neighbourhood. Hence both those spaces are orderable and they
glue together at $\{x_0\}$ to the orderable space $X$.
\end{proof}

Classification leads to the following simple corollaries.

\begin{wnioch}
In a connected regularly locally ordered space there exists at most one pair
of distinct points such that their removal does not disconnect the space.
\end{wnioch}

Note that in general the set of points not separating a connected
locally ordered space may be very big. Example~\ref{przy.punktyN} shows
that in a separable space it can be of cardinality continuum.

\begin{wnioch}
Every connected but not compact regularly locally ordered space is orderable.
\end{wnioch}

Several theorems, using notion of end-finiteness (``randendlich''), proven by Herrlich in
\cite{herrlich} can be easily derived from classifications presented here, since
loop-ordered spaces are certainly not end-finite. Below we present one example.
\begin{wnioch}[Theorem 2 from Chapter IV in \cite{herrlich}]
A connected locally ordered space is orderable if an only if
it is $T_3$ and end-finite.\end{wnioch}

\begin{wnioch}\label{loc-conn}
Every locally connected locally ordered Hausdorff space is a disjoint union
of some number of loop-ordered spaces and connected linearly ordered spaces.
It is then hereditarily normal ($T_5$) and locally compact.
\end{wnioch}
\begin{proof}
The first assertion comes straightforward from the decomposition onto connected
components which are connected regularly locally ordered spaces (see Lemma~\ref{lem.loccon}).
They are also $T_5$, so is their disjoint union.
\end{proof}

There is one more fact about locally ordered spaces and connectedness.
\begin{lem}\label{lem.connected} 
A connected subset of a regularly locally ordered space is a regularly locally ordered space.
\end{lem}
\begin{proof}
Fix connected set $C$, point $x\in C$, its neighbourhood $U$ with orderable closure
and assume $C\setminus\{x\}\neq\emptyset$ (singletons are obviously regularly locally ordered).
We claim that there exists a neighbourhood of $x$ in $C$ which is an interval in $U$.

First observe, that there cannot be an interval around $x$ not intersecting $C$.
Otherwise $x$ would be an isolated point in $C$.

We can approach $x$ by a net in $C\cap U$. Without loss of generality
we can assume that it is a decreasing net (consisting of elements greater than $x$ in $U$).
There cannot be a decreasing net of points in $U\setminus C$ approaching $x$,
since then there would be an interval containing points from $C$ with ends outside,
leading to separation of the connected set $C$
(Regularity guarantees that closure of such interval is contained in $U$).
Hence some nontrivial interval including $x$ is contained in $C$.

If there simultaneously exists an increasing net approaching $x$, we need to repeat
the reasoning to obtain that $x$ lies in the interior of the interval contained
in $C$.

We obtained that $C$, for every its point, contains an interval in a
neighbourhood from the atlas of orders, which is an orderable neighbourhood in $C$.
\end{proof}

Note that without the assumption of regularity connected components
may not be locally ordered. The spaces from Example~\ref{przy.nieLPPT.skl} and
\ref{przy.nLPPT.skT2} include such components. The second example is
$\TcompH$ proving the minimality of $T_3$ axiom in the above lemma.

Using the previous lemma we can formulate the following property.
\begin{lem}\label{clop-dec} 
Every regularly locally ordered space has such open cover that its every element
is closed and either a linearly ordered space or a loop-ordered space.
\end{lem}
\begin{proof} Denote the given regularly locally ordered space by $X$.
Let us start with decomposing $X$ into connected components. For every connected
subset of a regularly locally ordered space is regularly locally ordered, it
is either a linearly ordered space (possibly singleton) or loop-ordered space.

Loop-ordered component is always compact and open (Lemma~\ref{lem.clopl}).

Linearly ordered connected component $C$ is open unless it has an endpoint.
For non-open component we can consider ordered neighbourhoods of
its endpoints (assume they are disjoint for distinct endpoints of one component)
and naturally treat both of them together with $C$
as one linearly ordered space (tough not connected anymore). Since order on
the connected part of the neighbourhoods has to agree (after reversing if necessary)
with the order on $C$, there can be easily chosen one order on the union.
Denote this union be $U$ and fix an order on it.

We can find any small (with closure contained in $U$) closed-open neighbourhood
$V$ of $C$ (just separate it from ends of arbitrary interval covering $C$.
Then pick the maximal convex set $K$ in $U$
such that $C\subseteq K\subseteq V$ (it is just a union of all such convex sets).
It is closed-open and orderable (since convex) neighbourhood of $C$.
\end{proof}

\section{Compact and similar spaces}
Classification of connected regularly locally ordered spaces can be somehow extended
on a wider class of spaces, namely those for which exist \emph{tame} atlases.

We start with noticing a following fact.
\begin{thm}\label{thm.loop}
Union of an arbitrary family of loop-ordered subspaces contained in
a regularly locally ordered space is both closed and open.
\end{thm}
\begin{proof}
Openness is a consequence of Lemma~\ref{lem.clopl}.

Assume $x$ belongs to the closure of $A:=\bigcup_{i\in I}L_i$, where $L_i$
are loop ordered subspaces of $X$. Consider arbitrary ordered neighbourhood
$U$ of $x$. Define closed set $F:=\overline{A} \setminus U$.
Since loop ordered space cannot be embedded in a linearly ordered space,
$L_i\setminus U \neq \emptyset$ for every $i\in I$ and consequently $F\neq \emptyset$.
The intersection $U\cap L_i$ is a disjoint union of connected orderable sets, hence
consists of several open intervals in $U$.

By regularity we can find disjoint neighbourhoods $V_x$ and $V_F$
of $x$ and $F$, respectively. Without loss of generality we can assume $V_x$ is
an open interval in $U$.

If $x$ does not belong to $A$, there exists $i_0\in I$ such that
some interval from $L_{i_0}\cap U$ is contained in $V_x$ (such intervals are present
in every neighbourhood of $x$ for $x$ is in the closure of $A$).
Then the endpoints of this interval in $L_{i_0}$ belong to the closure of $V_x$
as well as to the set $F\subseteq V_F$ what is a contradiction.
\end{proof}

Having the above theorem we can extend classification of regularly locally ordered spaces.
\begin{thm}\label{thm.lind}
Every regularly locally ordered Lindel\"of space is a disjoint union
of at most countably many loop-ordered spaces and at most two orderable spaces.
\end{thm}
\begin{proof} Denote given regularly locally ordered Lindel\"of space by $X$.
Fix a closed-open cover of $X$ from Lemma~\ref{clop-dec} and choose
countable subcover $\{U_n\}_{n\in\N}$. Then, by defining sets
$V_n:=U_n\setminus\bigcup_{j<n}U_j$, for $n\in\N$, we obtain a closed-open cover consisting
of pairwise disjoint sets. Note that loop-ordered components were already disjoint with
all other sets from the cover $\{U_n\}_{n\in\N}$ hence they ware not modified when passing to
$\{V_n\}_{n\in\N}$.
Then each of the sets $V_n$ is either a loop-ordered space or a closed-open subspace
of a linearly ordered space. Since every open subspace of a linearly ordered space is
a disjoint union of orderable spaces, we have actually decomposed $X$ into
a disjoint union of at most countably many loop-ordered spaces and some number of
orderable spaces.
To finish the proof it suffices to observe that
arbitrary disjoint union of orderable spaces is in fact a union of
at most two such spaces (Lemma~\ref{lem.join} in the Appendix).
\end{proof}

\begin{wnioch}\label{wn-2AP}
Every second-countable regularly locally ordered space is homeomorphic to
a disjoint union of at most countably many unit circles and a subspace of the real line.
\end{wnioch}
\begin{proof}
Such space is metrizable (see \cite[4.2.8]{engel}), hence all the loop-ordered
components have to be homeomorphic to the unit circle. Second-countable linearly
ordered subspaces are embeddable into the real line (see \cite[6.3.2c]{engel}).
\end{proof}

In the case of compact spaces the presented characterisation is somehow simpler.
\begin{wnioch}\label{wn-comp}
Every compact Hausdorff locally ordered space is a disjoint union of finitely many
loop-ordered spaces and possibly a single compact orderable space.
\end{wnioch}
\begin{proof}
Applying the theorem for Lindel\"of spaces we obtain decomposition
into a disjoint union of loop-ordered spaces and (at most two)
linearly ordered spaces. Since each component of the union is open
and the space is compact, there are finitely many of them.
For each of them is closed, the linearly ordered components are compact.
Disjoint union of two compact linearly ordered spaces is a compact orderable space
for it is enough to treat every element of the first space
as smaller than any element of the second.
\end{proof}

Further we can observe.
\begin{wnioch}\label{wnT5}
Every compact Hausdorff locally ordered space is hereditarily normal ($T_5$).
\end{wnioch}
This is a special case of a more general fact on paracompact spaces.
\begin{lem}
Every paracompact Hausdorff space which admits an open cover
of hereditarily normal subsets is itself hereditarily normal.
\end{lem}
\begin{proof}
Fix two separated sets $A$ and $B$ (i.e. $\cl{A} \cap B = A\cap \cl{B} =\emptyset$)
in the given paracompact space $X$.
We will show that $A$ has a neighbourhood with closure not intersecting $B$.

Given an open cover of $X$ consisting of sets with hereditarily normal closures
(for the space is $T_3$ we can easily build such cover), we can
pick a locally finite refinement $\mathcal V_0$. 
Now focus on $\mathcal V_1:= \{V\in \mathcal V_0: V\cap A\neq \emptyset\}=\{V_j\}_{j\in J}$,
an open cover of $A$.

For $j\in J$, by hereditary normality of $\cl{V_j}$, we can find open set $U_j\subseteq V_j$
such that $\cl{U_j} \cap B = \emptyset$ and $U_j\supseteq A\cap V_j$.
Note that the collection $\mcU:=\{U_j\}_{j\in J}$ is an open cover of $A$
locally finite in the space $X$.

Take $U=\bigcup_{j\in J} V_j$. It is a neighbourhood of $A$, and since $\mcU$ is locally finite,
$\cl{U}\cap B = \bigcup_{j\in J} \cl{U_j} \cap B = \emptyset$.
\end{proof}
While the lemma is interesting on its own right, we focus on the following immediate corollary.
\begin{wnioch}
Every paracompact Hausdorff locally ordered space is hereditarily normal ($T_5$).

In particular every Lindel\" of regularly locally ordered space is $T_5$.
\end{wnioch}

Since in most cases we obtain normality of a space as a consequence of compactness,
paracompactness or higher separation axioms, the following question arise.
\begin{prob}
Is every normal locally ordered space hereditarily normal?
\end{prob}

\subsection*{Compact extensions and subspaces}

It is well known that any linearly ordered space has a linearly ordered compactification
-- even one extending the original order (cf. \cite[Problem 3.12.3(b)]{engel}).
Regularly locally ordered spaces are precisely those locally ordered ones
admitting a compactification; however, they may not admit a locally
ordered compactification. In general they may not even be embeddable in a paracompact
locally ordered space (see Examples \ref{przy.rat} and \ref{deska}).

Even under strong assumptions such as metrizability and local connectedness
a locally ordered space may not admit a locally ordered compactification.
An example can be an infinite disjoint union of unit circles. According to
the Theorem~\ref{thm.loop}, it would be closed in any bigger regularly locally ordered
space hence the latter could not be compact.

In fact, from the description of compact locally ordered spaces (Corollary~\ref{wn-comp}),
one can deduce the characterisation of all spaces admitting a~locally ordered compactification.
\begin{thm}
A topological space admits a locally ordered compactification if and only if
it is a disjointed union of a~suborderable space and finitely many loop-ordered spaces.
\end{thm}
\begin{proof}[Proof sketch]
Clearly, every subspace of a compact locally ordered space is of the above form.
The reverse implication follows from the fact that a suborderable space admitts
a linearly ordered compactification.
\end{proof}

There is one more fact related to compact locally ordered spaces, namely
\begin{prop}\label{prop.compsub}
Every compact subset of a Hausdorff locally ordered space is
a regularly locally ordered space.
\end{prop}
\begin{proof}
Let $K$ be a compact subspace of a Hausdorff locally ordered space.
Fix $x\in K$ and denote its linearly ordered neighbourhood by $U$.

Consider a decreasing family $\mathcal V$ of open intervals in $U$,
such that $\{x\}=\bigcap \mathcal V$.
Assume the set $\cl{V\cap K}\setminus U$ is nonempty for every $V\in \mathcal V$.
The family of such compact sets is decreasing, hence the intersection is nonempty.
A point from this intersection is contained in the closure of any neighbourhood of $x$,
what is a contradiction with the fact that the space is Hausdorff.

We obtained that, for some open interval $V\ni x$, the compact set $\cl{V\cap K}$
is contained in the linearly ordered subspace $U$, and hence it is orderable itself.
Therefore, $V\cap K$ is an orderable neighbourhood of $x$ in $K$.
\end{proof}
\begin{wnioch}
If a locally ordered Hausdorff space does not contain a~loop-ordered subspace,
then its every compact subset is orderable.
\end{wnioch}

\section{Examples}
The following section is a collection of all significant examples of
locally ordered spaces mentioned in the paper. We also note several topological properties of
presented spaces which are not main focus of this article.

Except from the last Examples (\ref{przy.nLPPT.skT2}, \ref{przy.rat}, \ref{przy.nielc} and \ref{deska})
all presented spaces are second-countable locally ordered spaces.


When defining a locally order topology we will often use the following fact, which is a
straight consequence of the axioms of topology.
\begin{lem} For a family $\left\{(X_i,\tau_i)\right\}_{i\in I}$ of topological spaces
such that for any two indices $i,j\in I$ and two sets $U\in\tau_i$ and $V\in\tau_j$ holds $U\cap V\in \tau_i\cap\tau_j$, there exists precisely one topology $\tau$ on $X:=\bigcup_{i\in I}X_i$
such that $\{X_i:i\in I\}$ is an open cover of $X$ and the induced topologies
are equal to the initial.
\end{lem}
The presented condition means nothing but that topologies on any two spaces from the cover coincide
on their intersection. In our case, when each of the spaces would be linearly ordered,
it would mostly follow from the fact that the intersection is a disjoint union
of open intervals with orders coinciding on each one alone.

The simple lemma presented below is a useful tool when comes to verifying
semiregularity of a locally ordered space.
\begin{lem} A locally ordered space is semiregular if and only if it is Hausdorff
and admits an atlas of orders consisting of sets each being interior of its closure.
\end{lem}

To make some constructions easier to understand, we will use special graphical
representation based on the following assumptions:
\begin{enumerate}
\item every line segment (possibly curved) denotes a set homeomorphic to an interval on the real line;
\item a neighbourhood of a point contains all close points within the horizontal line passing through;
\item a neighbourhood of a point lying on a dashed line consist of all points close to the line,
except the other points lying on the line;
\item when a point lies on a line segment, at least a part of its neighbourhood
is contained in that segment.
\end{enumerate}

Symbols $\omega$ and $\omega_1$ stand for the first countable
and the first uncountable ordinal, respectively.
We recall, that for any ordinal number $\lambda$ holds $\lambda=[0,\lambda)_\lambda$.

\begin{przy}[Line with doubled origin]\label{przy.nieT2}
The space of concern is obtained from the real line by adding additional point with the same
deleted neighbourhoods as the point $0$. In terms of the diagrams it can be drawn like below.
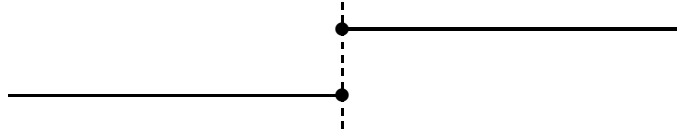
\setlength{\unitlength}{25pt}
\begin{center} \begin{picture}(10,2)(-5,-.5)
\multiput(0,-.5)(0,.2){10}{\line(0,1){.1}}
\thicklines
\put(0,1){\circle*{.2}}
\put(0,1){\line(1,0){5}}
\put(0,0){\circle*{.2}}
\put(0,0){\line(-1,0){5}}
\end{picture} \end{center}\medskip

The space is $T_1$, connected, locally connected, path connected but not Hausdorff
nor arcwise connected. Some further properties are listed in \cite{counter}.
\end{przy}

\setlength{\unitlength}{20pt}
\newsavebox{\vertd} \newsavebox{\vertdl}
\savebox{\vertd}(0,2)[b]{ \multiput(0,0)(0,-.2){10}{ \line(0,-1){.1} } }
\savebox{\vertdl}(0,3)[b]{ \multiput(0,0)(0,-.2){15}{ \line(0,-1){.1} } }

\begin{przy}\label{przy.nieUr}
Consider the set $X:=(0,\infty)_\R\cup\{a,b\}$ ($a\neq b$, $a,b\notin\R$)
and the following linearly ordered sets:
$$(0,\infty)_\R,\qquad
\bigcup_{n=1}^\infty (2n,2n+1)_\R \cup \{a\},\qquad
\bigcup_{n=1}^\infty (2n-1,2n)_\R \cup \{b\},$$
where real numbers are ordered naturally and points $a$, $b$ are the greatest elements
in the respective sets. The topologies on the intersections coincide,
hence the locally order topology is well defined.

The space $X$ is presented on the diagram below.
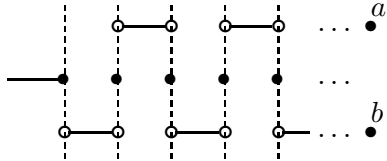
\begin{center}\begin{picture}(9,3)(-1,-1.5)
\multiput(-.1,-1.5)(1,0){5}{ \usebox{\vertdl} }  
\thicklines 
\multiput(.06,0)(1,0){5}{ \circle*{.2} } 
\put(.2,0){\line(-1,0){1}}
\matrixput(0.35,-1)(2,0){2}(1,2){2}{\line(1,0){.8}} 
\multiput(0.09,-1)(1,0){5}{ \circle{.2} } 
\put(4.16,-1){ \line(1,0){.5} }
\multiput(1.09,1)(1,0){4}{ \circle{.2} } 
\multiput(5,.92)(0,-1){3}{$\ldots$} 
\multiput(6,1)(0,-2){2}{\circle*{.2}} 
\put(6,1.2){$a$} \put(6,-.8){$b$}
\end{picture}\end{center} 

\begin{enumerate}[\ 1.]
\item Space $X$ is Hausdorff but no neighbourhoods of points $a$ and $b$
have disjoint closures.

\item After removing points $a$ and $b$, we are left
with a space homeomorphic to the real line.
Hence the space is $\sigma$-compact.

\item The space is semiregular, since the points from the middle row do not belong
to the interiors of ordered neighbourhoods of $a$ or $b$.
\end{enumerate}
\end{przy}

\begin{przy} \label{przy.nCompH}
Let $a:=(\omega_1,0)$ and $b:=(\omega_1,1)$. Consider the sets
$$A:=\omega_1\!\times\!(1/4,1/2)_\R \cup\{a\} \text{ and }
B:=\omega_1\!\times\!(3/4,1)_\R \cup\{b\}$$ with the lexicographic order.
For a limit oridnal $\lambda<\omega_1$ (also $0$) consider the set
$$U_\lambda:= \lambda\!\times\!(1/2,3/4)_\R \cup [\lambda,\lambda+\omega)_{\omega_1}\!\times\![0,1)_\R$$
with lexicographic order.
Order topologies on the intersections of given sets coincide, hence the locally order topology
on $X=\omega_1\!\times\![0,1)_\R\cup\{a,b\}$ is well defined.

\begin{enumerate}[\ 1.]
\item The above space $X$ is $\TcompH$ and semiregular.
\item $X$ is connected but not locally connected.
\item $X$ is not completely Hausdorff.

\begin{proof}
Suppose $f\colon X\!\to\![0,1]$ is continuous, and $f(a)\!=0$, $f(b)\!=1$.
There exists such $\lambda_0<\omega_1$ that for $\alpha\in(\lambda_0,\omega_1)_{\omega_1}$ holds
$f(\{\alpha\}\!\times\![1/4,1/2]_\R)\subseteq[0,1/3]_\R$ and $f(\{\alpha\}\!\times\![3/4,1]_\R)\subseteq[2/3,1]_\R$.

For a limit ordinal $\lambda\in(\lambda_0,\omega_1)_{\omega_1}$ any neighbourhood of a point
$(\lambda,0)$ contains a set $\{\alpha\}\!\times\!(1/2,3/4)_\R$
for some $\alpha\in(\lambda_0,\lambda)_{\omega_1}$.
The values of $f$ on the closure ($\{\alpha\}\!\times\![1/2,3/4]_\R$) are then contained in arbitrarily
small neighbourhood of the value $f((\lambda,0))$ (for sufficiently large $\alpha$).
This is a contradiction with $f((\alpha,1/2))\leq 1/3$ and $f((\alpha,3/4))\geq 2/3$.
\end{proof}
\end{enumerate}
\end{przy}

\begin{przy} \label{przy.nieT3} 
Consider the set $X:=[0,\infty)$ with topology given by the following basis of neighbourhoods:
for $x\in(0,\infty)$ we use euclidean neighbourhoods from $(0,\infty)$ and for $0$
we take sets of the form $\left[0,1/N\right)\cup\bigcup_{n=N}^\infty(2n,2n+1)$,
for natural $N\geq 1$.
The diagram is following:
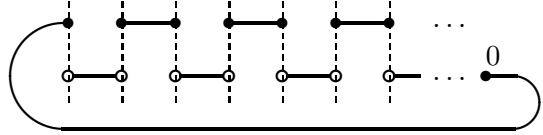
\begin{center} \begin{picture}(10,3)(-1,-1.5)
\multiput(-.17,-.5)(1,0){7}{ \usebox{\vertd} }
\thicklines 
\multiput(0,1)(1,0){7}{ \circle*{.2} }
\multiput(0,0)(1,0){7}{ \circle{.2} }
\multiput(1,1)(2,0){3}{ \line(1,0){1} }
\multiput(.27,0)(2,0){3}{\line(1,0){.8} }
\put(6.27,0){\line(1,0){.5}}
\put(8,.2){$0$}
\put(8,0){\circle*{.2}}
\put(-.13,-1){ \line(1,0){8.5} }
\put(7.9,0){ \line(1,0){.5} }
\multiput(7,.92)(0,-1){2}{$\ldots$}
\put(.1,0){\oval(2,2)[l]}
\put(8.5,-.5){\oval(1,1)[r]}
\end{picture} \end{center}

\begin{enumerate}[\ 1.]
\item Note, that after removing point $0$ we are left with a space homeomorphic
to the real line. Hence, the sapce is $\sigma$-compact.

\item $X$ is completely Hausdorff but not regular ($T_3$).

\item $X$ is arcwise connected but not locally connected.

\item The main idea behind this example is closely related to
\emph{Smirnov's deleted sequence topology} (see \cite{counter}),
also referred to as \emph{K-topology}.
\end{enumerate}
\end{przy}

\begin{przy}\label{przy.punktyN}
Fix an enumeration of rational numbers $\Q=\{q_n\}_{n=0}^\infty$.
Consider a space being the union of the following two linearly ordered spaces:

$A:=\bigl(\left(\RbezQ\right)\!\times\!\{0\}\bigr)\cup
\bigcup_{n=0}^\infty\bigl(\{q_n\}\!\times\!(2n,2n+1)_\R\bigr)$
with lexicographic order inherited from $\R\times\R$, and
$B:= \bigcup_{n=0}^\infty\bigl(\{q_n\}\!\times\!(2n,2n+2]_\R\bigr)$
ordered by the second coordinate.

$A$ can be viewed as space $\R$ modified by replacing each rational number
by an open interval (homeomorphic to $(0,1)$) and $B$ is homeomorphic to $(0,\infty)$.
On the intersection $A\cap B$ both topologies clearly agree,
hence global topology on $X=A\cup B$ is well defined.

\begin{enumerate}[\ 1.]
\item Since the set $B$ is dense in $X$, connected and separable,
the whole space $X$ is connected and separable.
\item $X$ is not regular nor locally connected.
\item $X$ is completely Hausdorff.

\item Consider the elements of $\left(\RbezQ\right)\!\times\!\{0\}\subseteq X$.
Neither of them belongs to the connected and dense subset $B$,
hence removing arbitrarily many of them does not separate $X$.
\end{enumerate}
\end{przy}

The following two spaces contain not locally ordered connected components.
\begin{przy}
\newcommand{\mjov}[1]{\frac{-1}{#1}}
\label{przy.nieLPPT.skl} We take the set $X=T\cup B\cup E$, where
\begin{align*}
T=
\bigcup_{n=0}^\infty \left[\textstyle\mjov{4n-3},\mjov{4n}\right]_\R&
\times\{1\},\qquad E=[0,1)_\R\times\{0\}, \\
B=\bigcup_{n=0}^\infty\textstyle\Bigl(\bigl[\mjov{8n},\mjov{8n+1}\bigr]_\R\Bigr.&
\textstyle \cup\bigl[\mjov{8n+2},\mjov{8n+3}\bigr)_\R\cup \\
\cup &\textstyle \Bigl.\bigl(\mjov{8n+4},\mjov{8n+5}\bigr)_\R
\cup\bigl(\mjov{8n+6},\mjov{8n+7}\bigr]_\R\Bigr)\times\{0\}.
\end{align*}
Endow the sets $B\cup E$ and
$T\cup\bigcup_{n=0}^\infty\bigl(\mjov{4n},\mjov{4n+1}\bigr)_\R\times\{1\}$
with order topology given by the natural order on the first coordinate.
Their topologies agree on the intersection,
hence the locally order topology on $X$ is well defined.

The idea is presented on the diagram below. Note, that the basic neighbourhoods
of the points on the lower level ($B\cup E$) do not include the points from above.
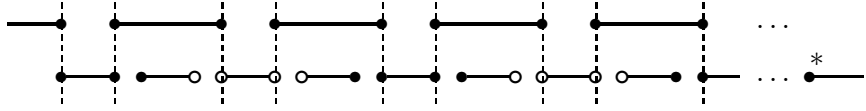
\begin{center} \begin{picture}(16,2)(-1,-.5)
\matrixput(-.17,-.5)(3,0){4}(1,0){2}{ \usebox{\vertd} }
\put(11.83,-.5){\usebox{\vertd}}
\thicklines 
\matrixput(0,0)(6,0){3}(0,1){2}{ \circle*{.2} }
\matrixput(1,0)(6,0){2}(0,1){2}{ \circle*{.2} }
\matrixput(3,1)(6,0){2}(1,0){2}{ \circle*{.2} }
\multiput(.83,1)(3,0){4}{ \line(1,0){2} } 
\matrixput(.1,0)(3,0){4}(1.5,0){2}{\line(1,0){.8} } 
\matrixput(2.5,0)(.5,0){2}(1.5,0){2}{\circle{.2}}
\matrixput(8.5,0)(.5,0){2}(1.5,0){2}{\circle{.2}}
\matrixput(1.5,0)(4,0){2}(6,0){2}{\circle*{.2}}
\put(0,1){\line(-1,0){1}}
\put(14,0){\circle*{.2}}
\put(14,.2){$\ast$}
\put(13.83,0){ \line(1,0){1} }
\put(11.9,0){ \line(1,0){.6} }
\multiput(13,.92)(0,-1){2}{$\ldots$}
\end{picture} \end{center}

\begin{enumerate}[\ 1.]
\item $X$ is not Hausdorff.
\item The connected component of the point $\ast=(0,0)$ does not contain the half-open intervals.
Hence neighbourhoods of $\ast$ in its component are not orderable,
similarly as in the Example~\ref{przy.nieLokPorz}.
\end{enumerate}
\end{przy}

\begin{przy}\label{przy.nLPPT.skT2}
Fix an enumeration of rationals $\Q=\{q_n:n\in\N_+\}$ and for every irrational
$x\in\RbezQ$ fix one strictly increasing sequence $(x(k))_{k\in\N_+}$ of indices such
that $\lim_{k\to\infty} q_{x(k)} = x$.
Consider the following subsets of $\R\times\R$
$$L:=\bigcup_{n=1}^\infty \bigl\{q_n\bigr\}\!\times\!
\left(\textstyle \frac{-1}{2n},\frac{-1}{2n+2}\right]_\R,$$
$$P_x:=\bigl(\{x\}\!\times\!\left[0,1\right]_\R\bigr)\cup
\bigcup_{k=1}^\infty\bigl\{q_{x(k)}\bigr\}\!\times\!
\left(\textstyle \frac{-1}{2x(k)},\frac{-1}{2x(k)+1}\right)_\R,
\text{ for } x\in\RbezQ,$$
$$D:=\bigl(\Q\!\times\!\{0\}\bigr) \cup \bigcup_{x\in\RbezQ} \{x\}\!\times\!
\bigl((0,1]_\R\cup[2,3)_\R\bigr).$$
Claim $L$ and each of $P_x$ ordered by the second coordinate
and on $D$ use lexicographic order inherited from $\R\times\R$.
It is easy to verify that order topologies coincide on intersections of any two
sheets, hence the locally order topology on $X:=L\cup D\cup\bigcup_{x\in\RbezQ}P_x$
is well defined.

\begin{enumerate}[\ 1.]
\item $X$ is not $T_3$.
\item $X$ is not separable.
\item It is a matter of routine to check that $X$ is completely Hausdorff 
and semiregular.
\item $X$ has a connected component which is not a locally ordered space.

\begin{proof}
Note that $L$ with considered order topology is naturally homeomorphic to $(-1,0)$, hence connected.
Each set $P_x$ is contained in the connected component of $L$. Since we can approach
a rational number $q$ with a sequence $(x_n)$ of irrationals, we can approximate the point
$(q,0)\in D$ with points $(x_n,1)\in P_{x_n}\cap D$. Hence $Q:=\Q\!\times\!\{0\}$ is also contained in the
same component as $L$. Since each of the sets $\{x\}\!\times\![2,3)_\R$, for irrational $x$, is both
closed and open in $X$, they do not belong to the component of $L$, which then appears to be
$C:=L\cup Q \cup \bigcup_{x\in\RbezQ} P_x $.

Fix a rational number $q_0$. We claim that the point $(q_0,0)\in X$ does not possess
any orderable neighbourhood in $C$.

A small neighbourhood $U$ of $(q_0,0)$ in $C$ is contained in
$D\cap C = Q \cup \bigcup_{x\in\RbezQ} \{x\}\!\times\!(0,1]_\R$.
Moreover, every such neighbourhood consists of uncountably many copies of $(0,1]$ and
countably many singletons (not isolated!). Actually $q_0$ belongs to the interior
of the set $U_1=$ projection of $U$ on the first coordinate.
Let $\tilde{U}$ denote an open interval in $\R$ contained in $U_1$.

There are only two possible orders on the open set $\{x\}\times(0,1]$ compatible with topology,
hence a point of the form $(x,1)$ has to be an extremal point of its neighbourhood.
Apart from at most two such points, every one has a successor or predecessor in hypothetical
order inducing the topology on $U$. Such ``neighbours''
can be only points of the form $(q,0)$, for rational $q$, or $(x,1)$, for irrational $x$.
Consider $G$ the set of all irrational numbers $x$ from $\tilde{U}$, for which there exists
other irrational number $x'\in\tilde{U}$ such that there is
no element between $(x,1)$ and $(x',1)$ in the order on $U$.
Note that $\tilde{U}\setminus G$ is countable.

To each $x\in G$ we can assign a positive number $r(x):=|x-x'|$.
For any rational number $q\in \tilde{U}$ and a sequence $(x_n)\subseteq G$ converging to $q$
holds $(x_n,1)\stackrel{U}{\to} (q,0)$, as well as $(x_n',1)\stackrel{U}{\to} (q,0)$,
hence $r(x_n)\to 0$.
Note that every such $q$ is a limit of some sequence in $G$.
For each $q\in \Q\cap \tilde{U}$ and $N\geq 1$ let $B(N,q)$ be a neighbourhood of $q$ such
that $r\bigl(G\cap B(N,q)\bigr)\subseteq (0,1/N)$. Then the sets
$B_N:=\bigcup\bigl\{B(N,q):q\in\Q\cap\tilde U\bigr\}$, for $N\geq 1$, have the intersection
contained in $\tilde{U}\setminus G$. This intersection would be a dense, countable
$G_\delta$ subset of $\tilde U$. From the Baire theorem we obtain contradiction.
\end{proof}
\end{enumerate}
\end{przy}

The last three examples are well known (though not necessarily for being locally ordered)
but we describe them briefly here to make this bank of examples more complete.
For more details on them the reader is referred to \cite{counter}.

\begin{przy}[Rational sequence topology] \label{przy.rat}
For every $x\in\RbezQ$ fix one sequence of rationals $(x_i)_{i\in\N}$ convergent to $x$.
Consider set $X:=\R$ with the following topology:
for $x\in\Q$ the singleton $\{x\}$ is open,
while for $x\in \RbezQ$ the basic neighbourhoods
have the form $\{x_i:i\geq n\}\cup\{x\}$, where $n\in\N$.

\begin{enumerate}[\ 1.]
\item All basic neighbourhoods are clearly closed and orderable hence the space
is regularly locally ordered. It is also locally compact.

\item The presented topology is finer than the standard topology on $\R$.
Hence, the diagonal ($\{(x,x):x\in X\}\subseteq X\times X$)
is a $G_\delta$ subset of the product space.

\item The space $X$ is not normal ($T_4$).
It can be proven by Jones' Lemma (see e.g. \cite{willard})
since the space is separable and contains discrete subset ($\RbezQ$) of cardinality continuum.
According to Corollary~\ref{wnT5} it does not admit a locally ordered compactification.

\item A curious observation is that on each of the ordered neighbourhoods we can consider natural
order induced from $\R$, obtaining that for any two sheets from the atlas
the orders coincide on the intersection. It shows that the existence of such
``neat'' atlas does not imply any further ``regularity'' of a locally ordered space.
See also the next example.
\end{enumerate}
\end{przy}

\begin{przy}[Pointed rational extension of $\R$]\label{przy.nielc}
For each $x\in \RbezQ$ take the set $\{x\}\cup\Q$ with the order topology inherited from $\R$.
Such covering defines a locally order topology on the set $X:=\R$,
since intersection of any two sheets equals $\Q$ (with standard topology)
and is open in both.

\begin{enumerate}[\ 1.]
\item The presented topology is finer than the standard topology on $\R$.
Hence $X$ it is completely Hausdorff.

\item $X$ is separable and first-countable but not second-countable.

\item One can notice, that in the described space the only connected sets are intervals with respect
to the classical order on $\R$ and the whole space is connected. However, no point
has a connected orderable neighbourhood.

\item $X$ is not semiregular. Picking any open interval in $\Q$, its closure
would be an interval in $\R$ and taking interior does not exclude all the irrationals.

\item By adding additional point ``$\infty$'' one can close the space into a ``circle''.
Then there would be no cut-point.
\end{enumerate}
\end{przy}

\begin{przy}[Dieudonn\' e plank] \label{deska}
Consider the set $\cl{L}:=\omega_1\times\mathbb Z\cup\{(\omega_1,0)\}$ with lexicographic order.
Note, that the only limit point is such order is the greatest one, $\infty:=(\omega_1,0)$.
Denote $L:=\cl{L}\setminus\{\infty\}$. Take the sets $[0,\omega]_{\omega_1}\times\{\lambda\}$, for
$\lambda\in L$, and $\{n\}\times \cl{L}$, for $n\in\omega$, with lexicographic orders.
Intersection of any two of them is either empty or contains single point from $\omega\times L$,
which is isolated in any of the considered orders. Hence, the locally order
topology on $X:=[0,\omega]_{\omega_1}\times \cl{L}\setminus\{(\omega,\infty)\}$
is well defined. It equals the topology inherited
from the product space $[0,\omega]_{\omega_1}\times \cl{L}$

\begin{enumerate}[\ 1.]
\item It follows straight from the definition that $X$ is regularly locally ordered.
\item $X$ is not locally compact.
\item $X$ is not normal, since the closed sets $A:=[0,\omega)_{\omega_1}\times\{\infty\}$
and $B:=\{\omega\}\times L$ do not admit disjoint neighbourhoods.
It can be observed, that any neighbourhood of $A$ has at most countable
complement in $\omega\times L$, while every neighbourhood of $B$ has
uncountable intersection with $\omega\times L$.

Furthermore, $X$ does not admit a locally ordered compactification.
\item $X$ is not separable nor first-countable.
\item The space above is homeomorphic to the classical Dieudonn\'e plank (see \cite{counter}).
One can use a bijective map from $L$ to $\omega_1\times\N\simeq \omega_1$,
such that it preserves the first coordinate.
Then it can be observed that the basic neighbourhoods from one construction
are open in the other and vice versa.
\end{enumerate}
\end{przy}

\begin{rem} In \cite{herrlich} Herrlich  mentioned the space konwn as ``Deleted Tychonoff plank''
(see \cite{counter}) as an example of not normal regularly locally ordered space.
We replaced the example with Dieudonn\' e plank, since its local orderability is more explicit.
\end{rem}

\section{Appendix}
The below theorem is very well known, but the presented topological proof
seems to be worth mentioning.
\begin{thm}\label{thm.con-ord}
For every connected linearly ordered space consisting of at least two points
there exist exactly two linear orders compatible with topology.
\end{thm}
\begin{proof}
Let $(X,<)$ be a linearly ordered set such that the induced order topology is connected.
The space $X\times X$ with product topology is then also connected. Moreover,
from connexity, $X\times X\setminus \{(x,x):x\in X\} = < \cup <^{-1}$.

Let $R\subseteq X\times X$ be a linear order compatible with topology.
For two arbitrary points $(x_1,y_1)$ and $(x_2,y_2)$ in $R$, the sets
$(\leftarrow,x_1]_R\times[y_1,\rightarrow)_R$ and $(\leftarrow,x_2]_R\times[y_2,\rightarrow)_R$
are contained in $R$ and connected. Their intersection is not empty, since it
contains the point $(\min_R(x_1,x_2),\max_R(y_1,y_2))$.
Hence $R$ is a connected subset of the product space and
has to be contained in one of the disjoint open sets
$<=\{(x,y):x<y\}$ and $<^{-1} =\{(x,y):y<x\}$.
Since no proper subset of a linear order is a linear order, either $R=<$ or $R=<^{-1}$.
\end{proof}

\begin{lem}\label{lem.join}
If a topological space is a disjoint union of an arbitrary family of totally
ordered spaces, then it can be obtained as a disjoint union of at most two
linearly ordered spaces.
\end{lem}
\begin{proof}
By concatenation of two given linearly ordered spaces we mean extending both orders to
an order on the union in such a way, that all elements of the first space are smaller than
elements of the second. To preserve the disjointness of the union we must be sure that
the first space contains the greatest element if and only if the second one has the smallest.

Let $\mathcal{U}$ be a given family of linearly ordered spaces. We divide it into three
disjoint subfamilies $\mathcal{U}_0$, $\mathcal{U}_1$ and $\mathcal{U}_2$, namely spaces
with no extremal point, with one extremal point (smallest or greatest element)
and with both extremal points respectively. 

We can easily concatenate any two spaces from $\mathcal{U}_1$ to obtain one linearly
ordered space with no endpoint (reversing one of the orders, if necessary). Proceeding
this way we can make sure that there is at most one element in $\mathcal{U}_1$.

Similarly, we can concatenate countably many spaces from $\mathcal{U}_2$ into one space
with no endpoint, hence we may reduce to the case when the family $\mathcal{U}_2$
consists of at most one element (any finite concatenation still has both extremal points).

The family $\mathcal{U}_0$ (under no assumptions on cardinality) can be also concatenated
to obtain one linearly ordered space by using sufficiently large ordinal. We are left in
the case when all three families are at most singletons. Since space with one extremal
point can be easily concatenated (after reversing the order, if necessary)
with any linearly ordered space, we are done.
\end{proof}
Herrlich \cite{herrlich} formulated a similar, but more specific theorem providing
sufficient conditions for orderability of a disjointed union.

\subsection*{Acknowledgements}
Most of the paper had been already finnished before the author got access to the dissertation
by Herrlich \cite{herrlich}.

The author would like to thank Piotr Niemiec for valuable suggestions at several stages of work.

\end{document}